\documentclass[12pt,leqno,oneside]{amsart}
\usepackage{mathrsfs,dsfont}
\usepackage{amsmath,amstext,amsthm,amssymb,amscd,bbm}
\usepackage{charter}
\usepackage{typearea}
\usepackage{pdfsync}
\usepackage{color}
\usepackage{mathtools}

\usepackage{hyperref}

\usepackage[width=6.4in,height=8.5in]
{geometry}

\numberwithin{equation}{section}

\newtheorem{Theorem}{Theorem}[section]
\newtheorem{Proposition}[Theorem]{Proposition}

\newtheorem{Corollary}[Theorem]{Corollary}
\theoremstyle{definition}

\newtheorem{Example}[Theorem]{Example}

\newcommand{\db}{\overline\partial}

\newcommand{\ddbar}{{\partial\overline\partial}}
\newcommand{\wi}{\widetilde}
\newcommand{\wih}{\widehat}

\DeclareMathOperator{\ric}{Ric}
\DeclareMathOperator{\codim}{codim}

\DeclareMathOperator{\dist}{dist}

\DeclareMathOperator{\Bs}{Bs}

\newcommand{\cali}[1]{\mathscr{#1}}
\newcommand{\cO}{\cali{O}}

\newcommand{\cC}{\cali{C}}

\newcommand{\field}[1]{\mathbb{#1}}

\newcommand{\R}{\field{R}}
\newcommand{\C}{\field{C}}
\newcommand{\N}{\field{N}}

\renewcommand{\P}{\field{P}}

\newcommand{\E}{\field{E}}

\newcommand{\mO}{\mathcal{O}}
\newcommand{\eq}{\mathrm{eq}}

\newcommand{\reg}{\mathrm{reg}}
\newcommand{\sing}{\mathrm{sing}}
\newcommand{\FS}{\mathrm{FS}}
\newcommand{\PSH}{\mathrm{PSH}}
\newcommand{\Vol}{\mathrm{Vol}}

\newcommand{\comment}[1]{}

\begin{document}

\title[Zeros of random holomorphic sections of big line bundles
]
{Zeros of random holomorphic sections of big line bundles with continuous metrics}

\author{Turgay Bayraktar} 
\thanks{T.\ Bayraktar is partially supported by T\"{U}B\.{I}TAK \& German DAAD Collaboration Grant ARDEB-2531/121N191}
\address{Faculty of Engineering and Natural Sciences, Sabanc{\i} University, \.{I}stanbul, Turkey}
\email{tbayraktar@sabanciuniv.edu}

\author{Dan Coman}
\thanks{D.\ Coman is supported by the NSF Grant DMS-2154273}
\address{Department of Mathematics, Syracuse University, Syracuse, NY 13244-1150, USA}
\email{dcoman@syr.edu}

\author{George Marinescu}
\address{Univerisit\"at zu K\"oln, Mathematisches institut,
Weyertal 86-90, 50931 K\"oln, Germany 
\newline\mbox{\quad}\,Institute of Mathematics `Simion Stoilow', 
Romanian Academy, Bucharest, Romania}
\email{gmarines@math.uni-koeln.de}
\thanks{G.\ Marinescu is partially supported
by the DFG funded project
CRC TRR 191 \textquoteleft Symplectic Structures in Geometry, Algebra
and Dynamics\textquoteright{} (Project-ID 281071066-TRR 191),
the DFG Priority Program 2265 \textquoteleft Random
Geometric Systems\textquoteright{} (Project-ID 422743078) and the
ANR-DFG project QuaSiDy (Project-ID 490843120).}

\author{Vi{\^e}t-Anh Nguy{\^e}n}
\address{Universit\'e de Lille, 
Laboratoire de math\'ematiques Paul Painlev\'e, CNRS U.M.R. 8524,  
\newline
\mbox{\quad}\,59655 Villeneuve d'Ascq Cedex, France}
\email{Viet-Anh.Nguyen@univ-lille.fr}
\thanks{V.-A. Nguyen is partially supported by the Labex 
CEMPI (ANR-11-LABX-0007-01), the project QuaSiDy (ANR-21-CE40-0016), and 
Vietnam Institute for Advanced Study in Mathematics (VIASM)}


\dedicatory{Dedicated to the memory of Steven Morris Zelditch}

\date{March 31, 2024}

\pagestyle{myheadings}

\begin{abstract}
Let $X$ be a compact normal complex space, $L$ be a big holomorphic line 
bundle on $X$ and $h$ be a continuous Hermitian metric on $L$. We consider 
the spaces of holomorphic sections $H^0(X, L^{\otimes p})$ endowed with the 
inner product induced by $h^{\otimes p}$ and a volume form on $X$,
and prove that the corresponding sequence of normalized Fubini-Study currents
converge weakly to the 
curvature current $c_1(L,h_\eq)$ of the equilibrium metric $h_\eq$ associated to $h$. 
We also show that the normalized currents of integration along the zero divisors of 
random sequences of holomorphic sections converge almost surely to $c_1(L,h_\eq)$, 
for very general classes of probability measures on $H^0(X, L^{\otimes p})$.
\end{abstract}

\maketitle
\tableofcontents

\section{Introduction} \label{introduction}

Let $X$ be a projective manifold and $(L,h)$ be a positive Hermitian holomorphic 
line bundle on $X$. One can define a sequence of Fubini-Study forms 
$\gamma_p$ on $X$ by setting $\gamma_p=\Phi_p^\star(\omega_\FS)$, where 
$\Phi_p:X\to\mathbb P\big(H^0(X,L^{\otimes p})^\star\big)$ is the Kodaira map 
associated to $(L^{\otimes p},h^{\otimes p})$ and $\omega_\FS$ is the Fubini-Study 
form on a projective space. A theorem of Tian \cite{Ti90} (see also \cite{Ru98}) 
states that $\frac{1}{p}\,\gamma_p\to c_1(L,h)$ as $p\to\infty$, in the $\cC^\infty$ 
topology on $X$. Tian's theorem is a consequence of the first term asymptotic 
expansion of the Bergman kernel associated to the inner product on 
$H^0(X,L^{\otimes p})$ determined by $h^{\otimes p}$ and a fixed volume form on 
$X$. The full asymptotic expansion of the Bergman kernel in this context was 
proved by Catlin \cite{Ca99} and Zelditch \cite{Z98}. Tian's theorem can be 
generalized to the case of singular Hermitian metrics $h$ with strictly positive 
curvature currents \cite{CM15}, and further to the case when $X$ is a compact 
analytic space \cite{CM13,CMM17}. The above convergence is now in the weak 
sense of currents on $X$. 

The case of smooth non-positively curved Hermitian metrics $h$ on line bundles 
$L$ over projective manifolds $X$ was first considered by Berman \cite{Ber09}. 
He introduced an equilibrium metric $h_\eq$ associated to $h$ and proved that 
the measures having Bergman kernels as density converge weakly to an 
equilibrium measure defined using the Monge-Amp\`ere operator and the weights 
of $h_\eq$ (see also the survey \cite{Z18}). When $X=\P^n$ is the projective 
space and $L=\cO(1)$ is the hyperplane bundle, the role of the equilibrium metric 
is played by the extremal plurisubharmonic functions \cite{Bl05,Bl09,BL15}.

In this paper we consider, more generally, line bundles endowed with arbitrary 
continuous Hermitian metrics over compact normal complex spaces which are 
not assumed to be K\"ahler. More precisely, our setting is as follows:

\smallskip

(A1) $X$ is a compact, reduced, irreducible, normal complex space of dimension $n$, 
and $\omega$ is a Hermitian form on $X$.

\smallskip

(A2) $L$ is a big holomorphic line bundle on $X$ and $h$ is a continuous Hermitian 
metric on $L$.

\smallskip

We will denote by $X_\reg$ the set of regular points of $X$ and by $X_\sing$ 
the set of singular points of $X$. The definition and some basic properties of 
big line bundles are recalled in Section \ref{S:big}. Throughout the paper, 
we let $d^c:= \frac{1}{2\pi i}\,(\partial -\overline\partial)$, so 
$dd^c=\frac{i}{\pi}\,\partial\overline\partial$.   

We fix a smooth Hermitian metric $h_0$ on $L$ and write 
\begin{equation}\label{e:varphi}
\alpha:=c_1(L,h_0)\,,\,\;h=h_0e^{-2\varphi}\,,
\end{equation}
where $\varphi\in L^1(X,\omega^n)$ is called the 
\textit{(global) weight of  $h$  relative to $h_0$}. The metric $h$ is called 
bounded, continuous, resp.\ H\"older continuous, if $\varphi$ is a bounded, 
continuous, resp.\ H\"older continuous, function on $X$. The definition of the 
curvature current $c_1(L,h)$ of a singular metric $h$ is recalled in 
Section \ref{S:big}. It is easy to see that 
\begin{equation}\label{e:cu1}
c_1(L,h)=\alpha+dd^c\varphi.
\end{equation}

In our present setting (A2), the function $\varphi$ is continuous. 
Following \cite{Ber09}, we let 
\begin{equation}\label{e:eqm}
h_\eq=h_0e^{-2\varphi_\eq}
\end{equation}
denote the \textit{equilibrium metric associated to $h$}, where 
\begin{equation}\label{e:eqw}
\varphi_\eq=V_\varphi=\sup\{u:\,u\in\PSH(X,\alpha,\varphi)\}
\end{equation} 
is the $\alpha$-plurisubharmonic envelope of $\varphi$ defined in \eqref{e:env}.
It is worth noting that similar envelopes in connection to quasi-plurisubharmonic 
functions with prescribed singularities along analytic subsets were constructed 
in \cite{RWN17,CMN22,CMN23,DX21}. We have
\begin{equation}\label{e:cu2}
c_1(L,h_\eq)=\alpha+dd^c\varphi_\eq.
\end{equation}
Moreover, observe that the equilibrium metric $h_\eq$ and the function 
\begin{equation}\label{e:psih}
\Psi_h=\varphi_\eq-\varphi
\end{equation}
are independent of the choice of $h_0$ and depend only on $h$. 
Note that $\Psi_h\in L^1(X,\omega^n)$ and $\Psi_h\leq0$ on $X$. 

For $p\geq1$, we let $L^p := L^{\otimes p}$ and $h^p:= h^{\otimes p}$ be the 
Hermitian metric on $L^p$ induced by $h$. Let 
$H^0_{(2)}(X,L^p)=H^0_{(2)}(X,L^p,h^p,\omega^n)$ 
be the Bergman space of $L^2$-holomorphic sections of $L^p$ relative to the 
metric $h^p$ and the volume form $\omega^n/n!$ on $X$, 
\begin{equation}\label{e:bs}
H^0_{(2)}(X,L^p)=\left\{S\in H^0(X,L^p):\,
\|S\|_p^2:=\int_X|S|^2_{h^p}\,\frac{\omega^n}{n!}<\infty\right\},
\end{equation}
endowed with the obvious inner product. Note that $H^0_{(2)}(X,L^p)=H^0(X,L^p)$, 
since the metric $h$ is continuous.

Let $P_p,\gamma_p$ be the Bergman kernel function and Fubini-Study current 
of $H^0_{(2)}(X,L^p)$, defined in \eqref{e:Bkf} and \eqref{e:FS}. 
Our first main result describes their asymptotic behavior as $p\to\infty$. Its proof is 
given in Section \ref{S:FS}.

\begin{Theorem}\label{T:FS}
Let $X,\omega,L,h$ verify assumptions (A1)-(A2), and $h_\eq,\Psi_h$ be defined 
in \eqref{e:eqm} and \eqref{e:psih}. Then, as $p\to\infty$, we have
\begin{equation}\label{e:cbk}
\frac{1}{2p}\,\log P_p\to\Psi_h \text{ in } L^1(X,\omega^n),\;
\frac{1}{p}\,\gamma_p\to c_1(L,h_\eq) 
\text{ weakly on $X$.}
\end{equation}
Moreover, if $h$ is H\"older continuous then there exist a constant $C>0$ and $p_0\in\N$ 
such that 
\begin{equation}\label{e:sbk}
\int_X\Big|\frac{1}{2p}\,\log P_p-\Psi_h\Big|\,\omega^n\leq C\,\frac{\log p}{p}\,,
\,\;\text{for all $p\geq p_0$}.
\end{equation}
\end{Theorem}

If $X$ is a projective manifold with an ample line bundle $L$ and $h$ is a 
H\"older continuous Hermitian metric on $L$, \eqref{e:sbk} was proved in 
\cite[Theorem 1.3]{DMM16}. Theorem \ref{T:FS} has analogues for the case of 
partial Bergman kernels corresponding to spaces of holomorphic sections that 
vanish to given orders along a collection $\Sigma$ of analytic subsets of $X$ \cite{CMN22,CMN23,RWN17}. It is worth noting that partial Bergman kernels 
have found applications to K\"ahler geometry and their asymptotic expansion 
was intensively studied (see, e.g., \cite{RS17,CM17,ZZ19,Fi21,Fi22} and 
references therein). 

To prove Theorem \ref{T:FS} we use ideas and techniques from 
the proof of the corresponding result for partial Bergman kernels 
\cite[Theorem 1.8]{CMN22}. In that case, one has to construct the equilibrium 
weight $\varphi_\eq$ using a divisorization $(\wi X,\pi,\wi\Sigma)$ of $(X,\Sigma)$ 
(see \cite[Proposition 1.4]{CMN22}), where $\pi:\wi X\to X$ is a desingularization 
that transforms $\Sigma$ to a divisor $\wi\Sigma\subset\wi X$. Then 
$\varphi_\eq$ is obtained as the push forward to $X$ of an upper envelope of 
quasi-plurisubharmonic functions on $\wi X$ with logarithmic poles along 
$\wi\Sigma$. In the present situation we are able to define the equilibrium weight 
$\varphi_\eq$ working directly on $X$. We develop in Section \ref{S:env} the 
required results from pluripotential theory on complex spaces. 

\medskip

In a series of papers starting with \cite{SZ99}, Shiffman and Zelditch applied 
Tian's theorem to describe the asymptotic distribution, as $p\to\infty$, of zeros 
of random holomorphic section in $H^0(X,L^p)$, where $L$ is an ample line 
bundle on a projective manifold $X$ and one uses the Gaussians or the area 
measure of unit spheres as probability measures on $H^0(X,L^p)$. Their 
fundamental work extends classical results on the distribution of zeros of random 
polynomials to the setting of projective manifolds and holomorphic sections. 

Our second main result deals with the asymptotic behavior of zeros of random 
sequences of sections, in the framework (A1)-(A2) and for very general classes 
of probability measures on $H^0_{(2)}(X,L^p)$, following the 
methods of \cite{BCM}. Let $d_p=\dim H^0_{(2)}(X,L^p)$ and 
$S_1^p,\dots,S_{d_p}^p$ be an orthonormal basis of $H^0_{(2)}(X,L^p)$. 
Using these bases, we identify the spaces $H^0_{(2)}(X,L^p)\simeq \C^{d_p}$ 
and endow them with probability measures $\sigma_p$ verifying the 
following moment condition: 

\smallskip

(B) There exist a constant $\nu\geq1$ and for every $p\geq1$ constants 
$C_p>0$ such that  
\[\int_{\C^{d_p}}\big|\log|\langle a,u\rangle|\,\big|^\nu\,d\sigma_p(a)
\leq C_p\,,\,\text{for any $u\in\C^{d_p}$ with $\|u\|=1$}\,.\] 

\smallskip


Let $[s=0]$ be the current of integration over the zero divisor of $s\in H^0(X,L^p)$.
The expectation current $\E[s_p=0]$ of the random variable 
$H^0_{(2)}(X,L^p)\ni s_p\mapsto[s_p=0]$ is defined by
\begin{equation}\label{e:expec}
\big\langle \E[s_p=0],\phi\big\rangle=
\int\limits_{H^0_{(2)}(X,L^p)}\big\langle[s_p=0],\phi\big\rangle\,d\sigma_p(s_p),
\end{equation}
where $\phi$ is a $(n-1,n-1)$ test form on $X$. Consider the product probability 
space
\begin{equation}\label{e:calH}
(\mathcal{H},\sigma)=
\left(\prod_{p=1}^\infty H^0_{(2)}(X,L^p),\prod_{p=1}^\infty\sigma_p\right). 
\end{equation}
We have the following theorem.

\begin{Theorem}\label{T:zeros}
Assume that $(X,\omega)$, $(L,h)$, $\sigma_p$ verify assumptions 
(A1), (A2), (B), and let $h_\eq,\Psi_h$ be as in \eqref{e:eqm}, \eqref{e:psih}. 
Then the following hold:

\smallskip

(i) If $\displaystyle\lim_{p\to\infty}C_pp^{-\nu}=0$ then 
$\displaystyle\frac{1}{p}\,\E[s_p=0]\to c_1(L,h_\eq)$ 
as $p\to \infty$, weakly $X$. 

(ii) If $\displaystyle\liminf_{p\to\infty}C_pp^{-\nu}=0$ then there exists an increasing 
sequence $p_j$ such that for $\sigma$-a.e.\ sequence $\{s_p\}\in\mathcal{H}$ we 
have, as $j\to\infty$, 
\[\frac{1}{p_j}\,\log|s_{p_j}|_{h^{p_j}}\to\Psi_h \text{ in } L^1(X,\omega^n),\;
\frac{1}{p_j}\,[s_{p_j}=0]\to c_1(L,h_\eq) \text{ weakly on $X$.}\]

(iii) If $\displaystyle\sum_{p=1}^{\infty}C_pp^{-\nu}<\infty$ then for 
$\sigma$-a.e.\ sequence $\{s_p\}\in\mathcal{H}$ we have, as $p\to\infty$, 
\[\frac{1}{p}\,\log|s_p|_{h^p}\to\Psi_h \text{ in } L^1(X,\omega^n),\;
\frac{1}{p}\,[s_p=0]\to c_1(L,h_\eq) \text{ weakly on $X$.}\]
\end{Theorem}

Theorem \ref{T:zeros} is proved in Section \ref{S:zeros}. 
We also recall there several general classes of probability measures,
described in \cite{BCM,BCHM}, which 
verify assumption (B) with explicit values for the constants $C_p$. We refer to 
\cite{BCM,CM15,CMM17,CMN16,CMN18,DMM16,Ba1,S08,SZ99,SZ08} 
and to the surveys \cite{BCHM,Z18,Ba2} for equidistribution results for holomorphic 
sections in various contexts.


 \section{Envelopes of quasi-plurisubharmonic functions on complex spaces}
 \label{S:env}
 
Let $X$ be a (reduced) complex space and denote by $X_\reg$, resp.\ $X_\sing$, 
the set of regular points, resp.\ singular points, of $X$.
A chart $(U,\iota,V)$ on $X$ 
is a triple consisting of an open set $U \subset X$, a closed complex space 
$V \subset G \subset \C^N$ in an open set $G$ of $\C^N$ and a biholomorphic 
map $\iota:\ U \to V$. We call $\iota:\ U \to G \subset \C^N$ a local 
embedding of $X$.

We refer to \cite{D85} for the notions of differential forms and currents on $X$. 
Recall that a continuous (resp.\ smooth) function on $X$ is a function 
$\varphi:X \to\C$ such that for every $x \in X$ there exist a local 
embedding $\iota:U\to G \subset \C^N$ with $x\in U$ and a continuous 
(resp.\ smooth) function $\tilde\varphi:G \to\C$ such that 
$\varphi|_U =\tilde\varphi\circ\iota$. If $X$ is compact, a function 
$\varphi:X \to\C$ is called H\"older continuous if there exists a finite cover of $X$ 
by open sets $U$ as above, such that $\varphi|_U$ is H\"older continuous with 
respect to the metric on $U$ induced by the Euclidean distance on $\C^N$. 

A (\textit{strictly}) \textit{plurisubharmonic} (\textit{psh}) function on $X$ is a function 
$\varphi:X\to[-\infty, \infty)$ that is not identically $-\infty$ on any open 
subset of $X$ and such that for every $x \in X$ there exist a local embedding 
$\iota:U\to G \subset \C^N$ with $x\in U$ and a (strictly) psh function 
$\tilde\varphi: G\to[-\infty,\infty)$ such that $\varphi|_U= \tilde\varphi\circ\iota$.   
We let $\PSH(X)$ denote the set of all psh functions on $X$. A function 
$\varphi:X\to[-\infty, \infty)$ is called \textit{weakly psh} if it is psh on $X_\reg$ and 
it is locally upper bounded near each point of $X_\sing$. A function 
$\varphi:X\to[-\infty, \infty)$ is called 
\textit{quasi-plurisubharmonic} (\textit{qpsh}) if it is 
locally the sum of a psh function and a smooth one. 

Let $\alpha$ be a smooth real $(1,1)$-form on $X$ that is locally $dd^c$-exact, 
i.e. for every $x\in X$ there exists an open set $U\subset X$ containing $x$ and 
a smooth function $\rho:U\to\R$ such that $\alpha=dd^c\rho$. 
Note that, by the $\partial\overline\partial$-lemma, every closed $(1,1)$-form is 
locally $dd^c$-exact on $X_\reg$. If 
$x\in X_\sing$ this condition means that, after possibly shrinking $U$, there exists 
a local embedding $\iota: U\to G \subset \C^N$ and a smooth function $\tilde\rho$ 
on $G$ such that $\tilde\rho\circ\iota=\rho$ on $U$ and $\alpha=dd^c\rho$ on 
$U\cap X_\reg$. A function $\varphi:X\to[-\infty, \infty)$ is called 
\textit{$\alpha$-plurisubharmonic} (\textit{$\alpha$-psh}) if, 
for every $x \in X$, there exists an 
open set $U\subset X$ containing $x$ and a smooth function $\rho$ on $U$ such 
that $\alpha=dd^c\rho$ and $\rho+\varphi$ is psh on $U$. We denote by 
$\PSH(X,\alpha)$ the set of all $\alpha$-psh functions on $X$. We have the 
following equivalent definition, which is obvious in the case when $X$ is smooth.

\begin{Proposition}\label{P:qpsh}
Let $\alpha$ be a smooth, real, locally $dd^c$-exact $(1,1)$-form on $X$. 
Given a function $\varphi:X\to[-\infty, \infty)$, the following are equivalent:

(i) $\varphi$ is $\alpha$-psh.

(ii) If $V\subset X$ is open and $\rho$ is a smooth function on $V$ such that 
$\alpha=dd^c\rho$ then $\rho+\varphi$ is psh on $V$.
\end{Proposition}

\begin{proof}
Clearly $(ii)$ implies $(i)$. Assume that $\varphi$ is $\alpha$-psh, $V\subset X$ is open and $\rho$ is a smooth function on $V$ such that $\alpha=dd^c\rho$. Let $x\in V$. Then there exists an open set $U\subset V$ with $x\in U$ and a smooth function $\rho_1$ on $U$ such that $\alpha=dd^c\rho_1$ and $\rho_1+\varphi$ is psh on $U$. The function $\chi=\rho-\rho_1$ is smooth on $U$ and $dd^c\chi=0$ on $U\cap X_\reg$, so $\chi$ is pluriharmonic on $U\cap X_\reg$. We infer by \cite[Theorem 1.10]{D85} that $\chi$ is psh on $U$, since it is continuous at the points of $U\cap X_\sing$. Hence $\rho+\varphi=\rho_1+\varphi+\chi$ is psh on $U$.
\end{proof}

The next proposition follows easily from Brelot's theorem \cite[Theorem 1.7]{D85}.

\begin{Proposition}\label{P:Brelot}
Assume that $X$ is locally irreducible, $Y\subset X$ is an analytic subset with empty interior, and $\alpha$ is a smooth, real, locally $dd^c$-exact $(1,1)$-form on $X$. If $\varphi$ is an $\alpha$-psh function on $X\setminus Y$ that is locally upper bounded in a neighborhood of each point of $Y$ then $\varphi$ has a unique extension to an $\alpha$-psh function $\varphi^\star$ on $X$, defined by 
\[\varphi^\star(x)=\limsup_{X\setminus Y\ni y\to x}\varphi(y),\;x\in Y.\]
\end{Proposition}

\begin{proof}
Let $V\subset X$ be an open set such that $\alpha=dd^c\rho$ for some smooth function $\rho$ on $V$. It follows by Proposition \ref{P:qpsh} that $v:=\rho+\varphi$ is psh on $V\setminus Y$. By \cite[Theorem 1.7]{D85}, $v$ has a unique extension to a psh function $v^\star$ on $V$, given by 
\[v^\star(x)=\limsup_{V\setminus Y\ni y\to x}v(y)=\rho(x)+\limsup_{X\setminus Y\ni y\to x}\varphi(y),\;x\in V\cap Y.\]
\end{proof}

If $\varphi:X\to[-\infty,+\infty)$ is an upper semicontinuous function we define  
\begin{align}
\PSH(X,\alpha,\varphi)&:=\{u\in\PSH(X,\alpha):
\,u\leq\varphi \text{ on } X\}, \label{e:class}\\
V_\varphi&:=\sup\{u:\,u\in\PSH(X,\alpha,\varphi)\}. 
\label{e:env}
\end{align}
The function $V_\varphi$ is called the
\textit{$\alpha$-psh envelope of $\varphi$}.
Provided that $\PSH(X,\alpha,\varphi)\neq\emptyset$, we now show that 
$V_\varphi$ is the largest $\alpha$-psh function dominated by $\varphi$ on $X$.

\begin{Proposition}\label{P:env}
Assume that $X$ is locally irreducible and $\alpha$ is a smooth, real,
locally $dd^c$-exact $(1,1)$-form on $X$. Let $\varphi$
be an upper semicontinuous function on $X$ such that the set
$\PSH(X,\alpha,\varphi)\neq\emptyset$.
Then $V_\varphi\in\PSH(X,\alpha,\varphi)$.
\end{Proposition}

\begin{proof}
Let $V_\varphi^\star$ denote the upper semicontinuous regularization
of $V_\varphi$ on $X$. Since $\varphi$ is upper semicontinuous,
we have $V_\varphi^\star\leq\varphi$ 
on $X$. Moreover, it is well known that $V_\varphi^\star$ is
$\alpha$-psh on $X_\reg$ (see e.g.\ \cite{Kl91,GZ05}).

Let $V_1$ denote the unique $\alpha$-psh extension of $V_\varphi^\star|_{X_\reg}$ 
to $X$ provided by Proposition \ref{P:Brelot}. Using again the upper semicontinuity 
of $\varphi$ we see that $V_1\leq\varphi$ on $X$. So $V_1\in\PSH(X,\alpha,\varphi)$ 
and hence $V_1\leq V_\varphi$ on $X$. The proposition follows if we show that 
$V_\varphi\leq V_1$. Note that $V_\varphi\leq V_\varphi^\star=V_1$
holds on $X_\reg$. 
Let $x\in X_\sing$ and $u\in\PSH(X,\alpha,\varphi)$.
Using Proposition \ref{P:Brelot} 
we infer that 
\[u(x)=\limsup_{X_\reg\ni y\to x}u(y)\leq
\limsup_{X_\reg\ni y\to x}V_\varphi^\star(y)=V_1(x).\]
Hence $V_\varphi\leq V_1$ holds on $X_\sing$.
\end{proof}

We remark that $\PSH(X,\alpha,\varphi)\neq\emptyset$ if $X$ is compact, 
$\varphi:X\to\R$ is bounded, and $\PSH(X,\alpha)\neq\emptyset$. Indeed, if 
$u\in\PSH(X,\alpha)$ then $u$ is bounded above since $X$ is compact, 
so $u-M\leq\varphi$ for some constant $M$. 

We assume next that $X$ is a compact complex space, $Y\subset X$ is an analytic 
subset with empty interior such that $X_\sing\subset Y$, $\wi X$ is a compact complex 
manifold and $\pi:\wi X\to X$ is a surjective holomorphic map such that 
$\pi:\wi X\setminus E\to X\setminus Y$ is a biholomorphism, where $E=\pi^{-1}(Y)$ has 
empty interior. If $\alpha$  is a smooth, real, locally $dd^c$-exact $(1,1)$-form on $X$, 
we let $\wi\alpha=\pi^\star\alpha$ and consider the maps 
\begin{align}
&\pi^\star:\PSH(X,\alpha)\to\PSH(\wi X,\wi\alpha),\;\pi^\star u=
u\circ\pi, \label{e:pull} \\
&\pi_\star:\PSH(\wi X,\wi\alpha)\to\PSH(X,\alpha),\;
\pi_\star\wi u(x)=\left\{\begin{array}{ll}
\wi u\circ\pi^{-1}(x) & \mbox{if } x\in X\setminus Y,\\
\displaystyle{\limsup_{X\setminus Y\ni y\to x}\wi u\circ\pi^{-1}(y)} & \mbox{if } x\in Y.
\end{array}\right. \label{e:push}
\end{align}
Note that the map $\pi_\star$ is well defined by Proposition \ref{P:Brelot}. 
We have the following:

\begin{Proposition}\label{P:blow}
In the above setting, assume further that $X$ is a locally irreducible. Then the maps $\pi^\star$, $\pi_\star$ defined in \eqref{e:pull}, \eqref{e:push} are bijective and $\pi^\star=(\pi_\star)^{-1}$. Moreover, if $\varphi$ is an upper semicontinuous function on $X$ such that $\PSH(X,\alpha,\varphi)\neq\emptyset$ and $\wi\varphi=\varphi\circ\pi$, then $\pi^\star V_\varphi=V_{\wi\varphi}$.
\end{Proposition}

\begin{proof}
If $\wi u\in\PSH(X,\wi\alpha)$, then $\pi^\star\pi_\star\wi u=\wi u$ since they are $\wi\alpha$-psh functions and are equal on $\wi X\setminus E$, so a.e.\ on $\wi X$.  If $u\in\PSH(X,\alpha)$, then $\pi_\star\pi^\star u=u$ on $X\setminus Y$ and hence on $X$, by the uniqueness assertion of Proposition \ref{P:Brelot}. For the last claim, using Proposition \ref{P:env} we infer that $\pi^\star V_\varphi\leq V_{\wi\varphi}$ on $\wi X$. Moreover, $\pi_\star V_{\wi\varphi}\leq\varphi$ holds on $X\setminus Y$, and hence on $X$ since $\varphi$ is upper semicontinuous. Hence $\pi_\star V_{\wi\varphi}\leq V_\varphi$ and the conclusion follows.
\end{proof}

\section{Big line bundles and Moishezon spaces}\label{S:big}

In this section we recall the definition and some properties of Moishezon spaces 
and big line bundles.

Let $X$ be a (reduced) compact irreducible complex space of dimension $n$, 
and $\mathcal K(X)$ be the field of meromorphic functions on $X$. The 
\textit{algebraic dimension} $a(X)$ of $X$ is defined as the transcendence 
degree of $\mathcal K(X)$ over $\C$ (see e.g.\ \cite[Definition 3.2]{U75}). 
Equivalently, $a(X)$ is the maximal number of algebraically independent 
meromorphic functions on $X$. If two compact irreducible complex spaces 
$X_1$, $X_2$ are bimeromorphically equivalent then their meromorphic 
function fields $\mathcal K(X_1)$, $\mathcal K(X_2)$ are isomorphic. 
Hence the algebraic dimension $a(X)$ is a bimeromorphic invariant.
A theorem of Thimm and Remmert shows that $a(X)\leq\dim X$ 
(see e.g.\ \cite[Theorem 3.1]{U75}).

One calls $X$ a \textit{Moishezon space} if $a(X)=\dim X$ \cite[Definition 3.5]{U75}. 
By a theorem of Moishezon \cite{Moi66}, $X$ is then bimeromorphically equivalent 
to a projective manifold. Combining this with Hironaka's theorems on resolution of 
singularities we obtain the following theorem that will be needed later. 

\begin{Theorem}\label{T:Moi}
If $X$ is a compact, irreducible, Moishezon space then there exists a connected 
projective manifold $\wi X$ and a surjective holomorphic map $\pi:\wi X\to X$, 
given as a composition of finitely many blow-ups with smooth center, such that 
$\pi:\wi X\setminus E\to X\setminus Y$ is a biholomorphism, where $Y$ is an 
analytic subset of $X$, $X_\sing\subset Y$, and $E=\pi^{-1}(Y)$ is a divisor 
with only normal crossings. Moreover, if $\codim X_\sing\geq2$ then 
$\codim Y\geq2$. 
\end{Theorem}

\begin{proof}
By Hironaka's theorem (see e.g.\ \cite[Theorem 13.2]{BM97}), there exists a compact 
complex manifold $\wih X$ and a surjective holomorphic map $\sigma_1:\wih X\to X$, 
given as a composition of finitely many blow-ups with smooth center, such that 
$\sigma_1:\wih X\setminus E_1\to X_\reg$ is a biholomorphism, where 
$E_1=\sigma_1^{-1}(X_\sing)$ is a divisor in $\wih X$. Since $a(\wih X)=a(X)$,
 it follows that $\wih X$ is a Moishezon manifold. By Moishezon's theorem 
 (\cite{Moi66}, \cite[Theorem 3.6]{U75}), there exists a projective 
manifold $\wi X$ and a surjective holomorphic map $\sigma_2:\wi X\to\wih X$, 
given as a composition of finitely many blow-ups with smooth center, such that 
$\sigma_2:\wi X\setminus E_2\to\wih X\setminus Y_1$ is a biholomorphism, 
where $Y_1\subset\wih X$ is an analytic subset with $\codim Y_1\geq2$ and 
$E_2=\sigma_2^{-1}(Y_1)$ is a divisor in $\wi X$. We let 
\[\pi=\sigma_1\circ\sigma_2:\wi X\to X,\;Y=\sigma_1(Y_1)\cup X_\sing,\; 
E=\pi^{-1}(Y)=E_2\cup\sigma_2^{-1}(E_1).\]
By Remmert's theorem, $\sigma_1(Y_1)$ is an analytic subset of $X$ and 
$\codim\sigma_1(Y_1)\geq2$. Applying Hironaka's embedded resolution of 
singularities theorem (see e.g.\ \cite[Theorem 1.6]{BM97}) to $(\wi X,E)$, we may 
assume that $E$ is a normal crossings divisor and the conclusion follows.
\end{proof}

Let next $X$ be a compact irreducible complex space of dimension $n$ 
and $L$ be a holomorphic line bundle on $X$. Recall that a Hermitian form on 
$X$ is a smooth $(1,1)$-form $\omega$ such that for every point $x \in X$ there 
exist a local embedding $\iota:\ U\ni x \to G \subset\C^N$ and a Hermitian form 
$\tilde\omega$ on $G$ with $\omega=\iota^\star\tilde\omega$ on $U \cap X_\reg$. 
We have that $\omega^n/n!$ gives locally an area measure on $X$. The notions 
of \textit{singular Hermitian metric} $h$ on $L$ and 
\textit{curvature current} $c_1(L,h)$ 
of $h$ are defined as in the case when $X$ is smooth \cite{D90}. If $e_U$ is a 
holomorphic frame of $L$ on some open set $U$ then $|e_U|_h=e^{-\varphi_U}$, 
where $\varphi_U\in L^1_{loc}(U)$ is called a \textit{local weight} of $h$, and 
$c_1(L,h)\mid_U=dd^c\varphi_U$. 

The volume of $L$ is defined by (see e.g.\ \cite[Definition 2.2.31]{La1})
\begin{equation}\label{e:vol}
\displaystyle\Vol(L)=\limsup_{p\to\infty}\frac{\dim H^0(X,L^p)}{p^n/n!}.
\end{equation}

The \textit{Kodaira-Iitaka dimension} $\kappa(X,L)$ 
of $L$ is defined as follows (see e.g.\ \cite[Definition 2.1.3]{La1}, 
\cite[Definition 5.1]{U75}). If $X$ is normal then
\[\kappa(X,L)=\max\big\{\dim\Phi_p(X):\,p\geq 1,\,\dim H^0(X,L^p)\geq 1\big\}.\]
Here $\Phi_p$ is the Kodaira map  
\begin{equation}\label{e:Kodaira}
\begin{split}
&\Phi_p:X\dashrightarrow \P(H^0(X,L^p)^\star)\,,\\
&\Phi_p(x)=
\{S\in H^0(X,L^p):\,S(x)=0\},\;x\in X\setminus \Bs(H^0(X,L^p))\,,
\end{split}
\end{equation}
where a point in  $\P(H^0(X,L^p)^\star)$ is identified with a hyperplane 
through the origin in $H^0(X,L^p)$ and 
\[\Bs(H^0(X,L^p))= \{x \in X :\,S(x) = 0,\,\forall\,S\in H^0(X,L^p)\}\] 
is the base locus of $H^0(X,L^p)$. If $X$ is not normal and $\sigma:X^\star\to X$ is 
the normalization of $X$ then one defines 
\[\kappa(X,L)=\kappa(X^\star,\sigma^\star L).\]

Let us recall some important properties of the Kodaira-Iitaka dimension. 
If $Y$ is a (reduced) compact irreducible complex space and 
$f:Y\dashrightarrow X$ is a bimeromorphic map then 
\begin{equation}\label{e:KIinv}
\kappa(Y,f^\star L)=\kappa(X,L).
\end{equation}
By \cite[Theorem 5.13]{U75}, the same conclusion holds if $f:Y\to X$ is a 
surjective holomorphic map. Moreover, one always has 
(see e.g.\ \cite[Lemma 5.5]{U75})
\begin{equation}\label{e:KIA}
\kappa(X,L)\leq a(X)\leq\dim X.
\end{equation}

A line bundle $L$ is called \textit{big} if $\kappa(X,L)=\dim X$ 
\cite[Definition 2.2.1]{La1}. One has that $L$ is big if and 
only if $\Vol(L)>0$, provided that $X$ is a compact complex 
manifold (see e.g.\ \cite[Theorem 2.2.7]{MM07}) or $X$ 
is a projective variety (see e.g. \cite[Section 2.2]{La1}). 
The following propositions summarize some basic properties 
of big line bundles.

\begin{Proposition}\label{P:big1}
Let $(X,\omega)$ be as in (A1) and $L$ be a holomorphic line bundle on $X$. 
The following are equivalent:

(a) $L$ is big\hspace{0.25mm};

(b) $\Vol(L)>0$\hspace{0.25mm};

(c) There exists a singular Hermitian metric $g$ on $L$ such that 
$c_1(L,g)\geq\varepsilon\omega$ for some constant $\varepsilon>0$. 
\end{Proposition}

\begin{proof} 
Recall that $\codim X_\sing\geq2$ since $X$ is normal. 
Hironaka's theorem (see e.g.\ \cite[Theorem 13.2]{BM97}) shows that there 
exists a compact complex manifold $\wih X$ and a surjective holomorphic 
map $\sigma:\wih X\to X$, given as a composition of finitely many blow-ups 
with smooth center, such that $\sigma:\wih X\setminus E\to X_\reg$ is a 
biholomorphism, where $E=\sigma^{-1}(X_\sing)$ is a divisor in $\wih X$. 

Since $X$ is normal and $\sigma:\wih X\setminus E\to X_\reg$ is a 
biholomorphism, we infer by Riemann's second extension theorem 
\cite[p.\,143]{GR84} that the map 
$\sigma^\star:H^0(X,L^p)\longrightarrow H^0(\wih X,\sigma^\star L^p)$ 
is an isomorphism. Thus $\Vol(\sigma^\star L)=\Vol(L)$. Moreover, 
it follows from \eqref{e:KIinv} that $L$ is big if and only if $\sigma^\star L$ 
is big. Thus $(a),(b)$ are equivalent since $\sigma^\star L$ is big if and 
only if $\Vol(\sigma^\star L)>0$.

We now show that $(a)$ implies $(c)$. Let $\wih\omega$ be a Hermitian form 
on $\wih X$ such that $\wih\omega\geq\sigma^\star\omega$. Since $\sigma^\star L$ 
is big, it follow by \cite{JS93} that there exists a singular Hermitian metric $\wih g$ on 
$\sigma^\star L$ such that $c_1(\sigma^\star L,\wih g)\geq\delta\wih\omega$ for 
some $\delta>0$. Consider the singular metric $g=(\sigma^{-1})^\star\wih g$ 
on $L|_{X_\reg}$. It verifies $c_1(L|_{X_\reg},g)\geq\delta\omega$. 
Let $U\subset X$ be an open set such that $L$ has a holomorphic frame $e_U$ 
on $U$ and there exists a local embedding $\iota: U\to G \subset \C^N$. 
If $z=(z_1,\ldots,z_N)\in G$, by shrinking $U$ we may assume that 
\[C_1\iota^\star dd^c\|z\|^2\leq\omega\leq C_2\iota^\star dd^c\|z\|^2\]
holds on $U$ for some constants $C_2>C_1>0$. Let $\varphi_U$ be the 
local weight of $g$ on $U\cap X_\reg$, so $|e_U|_g=e^{-\varphi_U}$ and $dd^c\varphi_U\geq\delta\omega$ on $U\cap X_\reg$. It follows that 
$\varphi_U-C_1\delta\iota^\star\|z\|^2\in\PSH(U\cap X_\reg)$. Since 
$X$ is normal we conclude by Riemann's second 
extension theorem for psh functions \cite[Satz\,4]{GR56} that 
$\varphi_U-C_1\delta\iota^\star\|z\|^2$ extends to a psh function on $U$. So  
\[dd^c\varphi_U\geq C_1\delta\iota^\star dd^c\|z\|^2\geq 
C_1C_2^{-1}\delta\omega\]
holds on $U$. Therefore $g$ extends to a singular Hermitian metric of $L$ 
on $X$ and $c_1(L,g)\geq\varepsilon\omega$ for some $\varepsilon>0$.

Finally we show that $(c)$ implies $(a)$. There exists a smooth Hermitian 
metric $\theta$ on $F=\cO_{\wih X}(-E)$ and $C>0$ such that 
$\Omega=C\sigma^\star\omega+c_1(F,\theta)$ is a Hermitian form on 
$\wih X$ and $\Omega\geq\sigma^\star\omega$ (see e.g.\ 
\cite[Lemma 1]{Moi67}, \cite[Lemma 2.2]{CMM17}). Fix $m\in\N$ such that 
$m\varepsilon\geq C$. Then 
\[c_1(\sigma^\star L^m\otimes F,\sigma^\star g^m\otimes\theta)=
m\sigma^\star c_1(L,g)+c_1(F,\theta)\geq
m\varepsilon\sigma^\star\omega+\theta\geq\Omega,\]
hence by \cite{JS93}, the line bundle $\sigma^\star L^m\otimes F$ is big.
Note that $H^0(\wih X,(\sigma^\star L^m\otimes F)^p)$ is isomorphic to the 
subspace of sections in $H^0(\wih X,\sigma^\star L^{mp})$ that vanish to order 
at least $p$ along $E$. We infer that $\sigma^\star L^m$, and hence 
$\sigma^\star L$, are big line bundles. Thus $L$ is big, by \eqref{e:KIinv}.
\end{proof}
\begin{Proposition}\label{P:big2}
Let $X$ be as in (A1) and $L$ be a big holomorphic line bundle on $X$. 
Then $X$ is a Moishezon space and 
\[\Vol(L)=\lim_{p\to\infty}\frac{\dim H^0(X,L^p)}{p^n/n!}.\]
Moreover, if $\pi:\wi X\to X$ is the map from Theorem \ref{T:Moi} then 
$\pi^\star:H^0(X,L^p)\longrightarrow H^0(\wi X,\pi^\star L^p)$
is an isomorphism. 
\end{Proposition}
\begin{proof}
It follows from \eqref{e:KIA} that $\kappa(X,L)=a(X)=\dim X$, so $X$ is a 
Moishezon space. By \eqref{e:KIinv} $\kappa(\wi X,\pi^\star L)=\kappa(X,L)$, 
hence $\pi^\star L$ is big. Moreover, $\pi^\star L=\cO_{\wi X}(D)$ 
for some integral divisor $D$ in $\wi X$, since $\wi X$ is projective. By 
\cite[Example 11.4.7]{La2} (see also \cite{Bo02}),
\[\Vol(\pi^\star L)=
\lim_{p\to\infty}\frac{\dim H^0\big(\wi X,\pi^\star L^p\big)}{p^n/n!}.\]
Since $X$ is normal we have that $\codim X_\sing\geq 2$, so 
$\codim Y\geq 2$. As $\pi:\wi X\setminus E\to X\setminus Y$ is a 
biholomorphism, where $E,Y$ are as in Theorem \ref{T:Moi}, 
Riemann's second extension theorem 
\cite[p.\,143]{GR84} implies that the map 
$\pi^\star:H^0(X,L^p)\longrightarrow H^0(\wi X,\pi^\star L^p)$
is an isomorphism. The conclusion follows.
\end{proof}
Note that in order to conclude that $X$ is Moishezon
in Proposition \ref{P:big2} we do not need $X$ to be normal
(required in (A1)), see \cite[Proposition 2.3]{CMM17}.
 
We end this section with a brief discussion of big cohomology classes 
on a compact complex manifold $X$. 
We consider the $\ddbar$-cohomology and particularly the 
space $H^{1,1}_\ddbar(X,\R)$, which is finite dimensional. If 
$\alpha$ is a smooth real closed $(1,1)$-form on $X$ we denote its 
$\ddbar$-cohomology class by $\{\alpha\}_\ddbar$. Recall that if $X$ is K\"ahler 
then by the $\ddbar$-lemma $H^{1,1}_\ddbar(X,\R)=H^{1,1}(X,\R)$, and we write 
$\{\alpha\}_\ddbar=\{\alpha\}$.

A class $\{\alpha\}_\ddbar$ is called \textit{big}
if it contains a \textit{K\"ahler current}, 
i.e., a positive closed current $T$ of bidegree $(1,1)$ on $X$ such that 
$T\geq \varepsilon\omega$ for some number $\varepsilon>0$, where 
$\omega$ is a Hermitian form on $X$. In this case, by Demailly's regularization 
theorem \cite[Proposition 3.7]{D92} (see also \cite[Theorem 3.2]{DP04}), 
one can find a K\"ahler current $T\in\{\alpha\}_\ddbar$ of the form 
\begin{equation}\label{e:Dreg}
T=\alpha+dd^c\varphi, \text{ where } \varphi\in\PSH(X,\alpha),\; 
\varphi=c\log\Big(\sum_{j=1}^\infty |g_j|^2\Big)+O(1) \text{ locally on } X,
\end{equation}
$c>0$ is a rational number and $g_j$ are holomorphic functions.
Note that the series $\sum_{j=1}^\infty |g_j|^2$ converges locally uniformly 
and its sum is a real analytic function.

The \textit{non-K\"ahler locus} of $\{\alpha\}_\ddbar$ is defined in 
\cite[Definition\ 3.16]{Bo04} as the set 
\[E_{nK}\big(\{\alpha\}_\ddbar\big)=
\bigcap\big\{E_+(T):\,\text{$T\in\{\alpha\}_\ddbar$ K\"ahler current}\big\},\] 
where $E_+(T)=\{x\in X:\,\nu(T,x)>0\}$ and $\nu(T,x)$ is the Lelong number 
of $T$ at $x$. By \cite[Theorem\ 3.17]{Bo04}, there exists a 
K\"ahler current $T=\alpha+dd^c\varphi$ as in \eqref{e:Dreg} such that 
\begin{equation}\label{e:EnK}
E_{nK}\big(\{\alpha\}_\ddbar\big)=E_+(T)=\{\varphi=-\infty\}.
\end{equation}
In particular, $E_{nK}\big(\{\alpha\}_\ddbar\big)$ is an analytic subset of $X$.

\section{Convergence of the Fubini-Study currents}\label{S:FS}

We start by recalling the definition and some basic facts about Bergman kernels and Fubini-Study currents. Then we give the proof of Theorem \ref{T:FS}. 

\subsection{Bergman kernels and Fubini-Study currents}

Let $X,\omega$ be as in (A1) and $(L,h)$ be a singular Hermitian holomorphic line bundle on $X$. Since $X$ is 
compact, the space $H^0(X,L)$ is finite dimensional. Let 
 \[H^0_{ (2)}(X, L) = H^0_{ (2)}(X,L,h, \omega^n)\] 
 be the Bergman space of $L^2$-holomorphic sections of $L$ 
 relative to the metric $h$ and the volume form $\omega^n /n!$ 
 on $X$, endowed with the inner product
\begin{equation}\label{e:iprod}
(S,S'):=\int_X\langle S,S'\rangle_h\,\frac{\omega^n}{n!}\,.
\end{equation}
Set $\|S\|^2=\|S\|^2_{h,\omega^n}:=(S,S)$. 

Let $r=\dim H^0_{ (2)}(X, L)$ and $S_1,\ldots,S_r$ be
an orthonormal basis of $H^0_{ (2)}(X, L)$. 
The \textit{Bergman kernel function} $P$ of $H^0_{ (2)}(X, L)$ 
is defined by 
\begin{equation}\label{e:Bkf}
P(x)=\sum_{j=1}^r|S_j(x)|_h^2,\;\;|S_j(x)|_h^2
:=\langle S_j(x),S(x)\rangle_h,\;x\in X.
\end{equation}
The definition of $P(x)$ does not depend on the choice
of the orthonormal basis. 
If $U\subset X$ is an open set where $L$ has a local holomorphic frame $e_U$, then 
$|e_U|_h=e^{-\varphi_U}$, $\varphi_U\in L^1_{loc}(U,\omega^n)$, and $S_j=s_je_U$, 
$s_j\in\cO_X(U)$. It follows that 
\begin{equation}\label{e:Bk_local}
\log P\mid_U= \log\Big(\sum_{j=1}^r|s_j|^2 \Big)-2\varphi_U,
\end{equation}
so $\log P\in L^1(X,\omega^n)$. When $h$ is bounded we have that $P$ satisfies 
the following variational principle:
\begin{equation}\label{e:Bkvar}
P(x)= \max \left\{|S(x)|^2_h:\,S\in H^0_{(2)}(X, L),\,\|S\| = 1\right\},\;x\in X.
\end{equation}

The \textit{Fubini-Study current} $\gamma$ of $H^0_{ (2)}(X, L)$ is defined by
\begin{equation}\label{e:FS}
\gamma\mid_U = \frac{1}{2}\,dd^c \log\Big(\sum_{j=1}^r |s_j|^2\Big),
\end{equation}
where  $U$ and $s_j$ are as above. Then $\gamma$ is a positive closed current 
of bidegree $(1,1)$ on $X$ and, by \eqref{e:Bk_local},
\begin{equation}\label{e:Bk_FS}
\gamma=c_1 (L, h)+ \frac{1}{2}\,dd^c \log P.
\end{equation}
Note that $\gamma=\Phi^\star\omega_\FS$, where $\Phi$ is the Kodaira map 
defined in \eqref{e:Kodaira} for the space $H^0_{ (2)}(X, L)$, and $\omega_\FS$ 
denotes the Fubini-Study form on a projective space $\P^{r-1}$.

We conclude this section by recalling the Lelong-Poincar\'e formula: If $[s=0]$ is 
the current of integration along the zero divisor of $s\in H^0(X,L)$ then 
$\log|s|_h\in L^1(X,\omega^n)$ and 
\begin{equation}\label{e:LP}
[s=0]=c_1 (L, h)+dd^c\log|s|_h.
\end{equation}

\subsection{Proof of Theorem \ref{T:FS}}\label{SS:FS}

Since $X$ is a normal space, it is locally irreducible and we have that 
$\codim X_\sing\geq2$. Moreover, by Proposition \ref{P:big2} we have 
that $X$ is a Moishezon space and $\dim H^0(X,L^p)\sim p^n$.

Let $h_0,\varphi$ be as in \eqref{e:varphi}, and $P_p,\gamma_p$ 
be the Bergman kernel function and Fubini-Study current of the 
space $H^0_{(2)}(X,L^p)=H^0(X,L^p)$. Using \eqref{e:Bk_FS} and 
\eqref{e:cu1} we obtain   
\begin{equation}\label{e:FSp1}
\frac{1}{p}\,\gamma_p=c_1(L,h)+\frac{1}{2p}\,dd^c\log P_p=
\alpha+dd^c\varphi_p\,, 
\end{equation}
where
\begin{equation}\label{e:FSp2}
\varphi_p=\varphi+\frac{1}{2p}\,\log P_p\in\PSH(X,\alpha).
\end{equation}

We call the function $\varphi_p$ the \textit{global Fubini-Study potential} 
of $\gamma_p$ (with respect to the fixed Hermitian metric $h_0$). 
Since $\varphi$ is continuous we have $\varphi_p-M\in\PSH(X,\alpha,\varphi)$, 
for some $p$ such that $H^0(X,L^p)\neq\emptyset$ and some constant $M$, 
where $\PSH(X,\alpha,\varphi)$ is defined in \eqref{e:class}. Thus the function 
$\varphi_\eq=V_\varphi$ from \eqref{e:eqw} is well defined and 
$\varphi_\eq\in\PSH(X,\alpha,\varphi)$ by Proposition \ref{P:env}. 

Theorem \ref{T:Moi} shows that there exists a projective manifold $\wi X$ 
and a surjective holomorphic map $\pi:\wi X\to X$ such that 
$\pi:\wi X\setminus E\to X\setminus Y$ is a biholomorphism, where 
$Y\supset X_\sing$ is an analytic subset of $X$, $\dim Y\leq n-2$, and 
$E=\pi^{-1}(X_\sing)$ is a normal crossings divisor. 
Since $\wi X$ is projective we can find a K\"ahler form $\wi\omega$ 
on $\wi X$ such that 
\begin{equation}\label{e:wiom}
\wi\omega\geq\pi^\star\omega,
\end{equation} 
and we denote by $\dist$ the distance on $\wi X$ induced by $\wi\omega$. Set 
\begin{equation}\label{e:wivarphi}
\wi L:=\pi^\star L\,,\,\;\wi h_0:=\pi^\star h_0\,,\,\;\wi\alpha:=
\pi^\star\alpha=c_1(\wi L,\wi h_0)\,,\,\;\wi\varphi:=
\varphi\circ\pi\,,\,\;\wi h:=\pi^\star h=\wi h_0e^{-2\wi\varphi}.
\end{equation}
As before, we use the notation $\wi h^p=\wi h^{\otimes p}$ and $\wi h_0^p=
\wi h_0^{\otimes p}$.

It follows by assumption (A2) and \eqref{e:KIinv} that the line bundle $\pi^\star L$ 
is big. Hence the class $\{\wi\alpha\}$ is big \cite{JS93}. By \eqref{e:EnK}, there 
exist $\eta\in\PSH(\wi X,\wi\alpha)$ as in \eqref{e:Dreg} and 
$\varepsilon_0>0$ such that
\begin{equation}\label{e:eta}
Z:=E_{nK}\big(\{\wi\alpha\}\big)=\{\eta=-\infty\},\;\eta\leq-1,\;
\wi\alpha+dd^c\eta\geq\varepsilon_0\wi\omega\geq
\varepsilon_0\pi^\star\omega\,
\text{ on } \wi X.
\end{equation}
 
Let $\wi\varphi_\eq=V_{\wi\varphi}$ be the $\wi\alpha$-psh envelope of 
$\wi\varphi$ on $\wi X$ defined in \eqref{e:env}. Proposition \ref{P:env}
shows that $\wi\varphi_\eq\in PSH(\wi X,\wi\alpha,\wi\varphi)$. Since 
$\eta-M\leq\wi\varphi$ for some constant $M>0$, we infer that 
$\wi\varphi_\eq$ is locally bounded on $\wi X\setminus Z$. 
By Proposition \ref{P:blow}, 
\[\pi^\star\varphi_\eq=\wi\varphi_\eq\,,\,\;\pi_\star\wi\varphi_\eq=\varphi_\eq,\]
where $\pi^\star,\pi_\star$ are defined in \eqref{e:pull}, \eqref{e:push}. Moreover, 
\begin{equation}\label{e:tpsih}
\wi\Psi_{\wi h}:=\wi\varphi_\eq-\wi\varphi=\Psi_h\circ\pi\leq0 \text{ on } \wi X.
\end{equation}
If $S\in H^0_{(2)}(X,L^p)$ it follows since $\pi:
\wi X\setminus E\to X\setminus Y$ is a biholomorphism that 
\[\int_X|S|^2_{h^p}\,\omega^n=
\int_{\wi X}|\pi^\star S|^2_{\wi h^p}\,\pi^\star\omega^n.\]
Using Proposition \ref{P:big2} we infer that the map 
\begin{equation}\label{e:iso}
S\in H^0_{(2)}(X,L^p)\to\pi^\star S\in H^0_{(2)}(\wi X,\wi L^p):=
H^0_{(2)}(\wi X,\wi L^p,\wi h^p,\pi^\star\omega^n)
\end{equation}
is an isometry. It follows that 
\begin{equation}\label{e:wiFS}
\wi P_p=P_p\circ\pi\,,\,\;\wi\gamma_p=\pi^\star\gamma_p
\end{equation}
are the Bergman kernel function, resp.\ Fubini-Study current, 
of the space $H^0_{(2)}(\wi X,\wi L^p)$. Note that  
\begin{equation}\label{e:wiFSp}
\frac{1}{p}\,\wi\gamma_p=\wi\alpha+dd^c\wi\varphi_p\,,\,
\text{ where }\,\wi\varphi_p=\wi\varphi+\frac{1}{2p}\,\log\wi P_p=
\varphi_p\circ\pi\in\PSH(\wi X,\wi\alpha).
\end{equation}
Theorem \ref{T:FS} will follow from our next result, which deals with the 
convergence of the Bergman kernels of $H^0_{(2)}(\wi X,\wi L^p)$.

\begin{Theorem}\label{T:wiFS}
Let $X,\omega,L,h$ verify assumptions (A1)-(A2), $\pi:\wi X\to X$ be as in 
Theorem \ref{T:Moi}, and $\varphi,\wi\Psi_{\wi h}$ be defined in \eqref{e:varphi}, 
\eqref{e:tpsih}. Then the following hold:

(i) $\displaystyle\frac{1}{2p}\,\log\wi P_p\to\wi\Psi_{\wi h}$ in 
$L^1(\wi X,\wi\omega^n)$ and locally uniformly on $\wi X\setminus Z$ 
as $p\to\infty$, where $\wi\omega,Z$ are defined in \eqref{e:wiom},\eqref{e:eta}.

(ii) If $\varphi$ is H\"older continuous on $X$ 
then there exist a constant $C>0$ and $p_0\in\N$ such that 
for all $x\in\wi X\setminus Z$ and $p\geq p_0$ we have 
\[\Big|\frac{1}{2p}\,\log\wi P_p-\wi\Psi_{\wi h}\Big|\leq
\frac{C}{p}\,\big(\log p+\big|\log\dist(x,Z)\big|\big).\]
\end{Theorem}

We follow the proof of \cite[Theorem 5.1]{CMN22} and adapt the arguments to 
work with the full Bergman kernel $\wi P_p$ instead of the partial Bergman 
kernel considered there (see also \cite{Ber09}, \cite{CM15}, \cite{RWN17} 
for similar approaches). We will need the following version of Demailly's 
$L^2$-estimates for the $\db$ equation \cite[Th\'eor\`eme 5.1]{D82} 
(see also \cite[Theorem 5.5]{CMN22}).

\begin{Theorem}\label{T:dbar}
Let $M$ be a complete K\"ahler manifold of dimension $n$, 
$\Omega$ be a (not necessarily complete) K\"ahler form on $M$, 
$\chi$ be a qpsh function on $M$, and $(F,h)$ be a singular 
Hermitian holomorphic line bundle on $M$. Assume that there exist 
constants $A,B>0$ such that 
\[\ric\Omega\geq-2\pi B\Omega\,,\,\;dd^c\chi\geq
-A\Omega\,,\,\;c_1(F,h)\geq(1+B+A/2)\Omega\,.\]
If $g\in L_{0,1}^2(M,F,{\rm loc})$ satisfies $\db g=0$ 
and $\int_M|g|^2_he^{-\chi}\,\Omega^n<+\infty$, 
then there exists $u\in L_{0,0}^2(M,F,{\rm loc})$ with 
$\db u=g$ and $\int_M|u|^2_he^{-\chi}\,\Omega^n\leq
\int_M|g|^2_he^{-\chi}\,\Omega^n$.
\end{Theorem}

\begin{proof}[Proof of Theorem \ref{T:wiFS}]
Let 
\[\Omega_{\wi\varphi}(\delta)=\sup\big\{|\wi\varphi(x)-\wi\varphi(y)|:\,
x,y\in\wi X,\;\dist(x,y)<\delta\big\}\]
be the modulus of continuity of $\wi\varphi$. We divide the proof in three steps.

\medskip

\textit{Step 1}. We show here that there exists a constant $C>0$ 
such that for all $p\geq1$ and $\delta\in(0,1)$ the following estimate holds on $\wi X$:
\begin{equation}\label{e:uest1}
\frac{1}{2p}\,\log\wi P_p\leq\wi\Psi_{\wi h}+C\Big(\delta+\frac{1}{p}-
\frac{\log\delta}{p}\Big)+2\Omega_{\wi\varphi}(C\delta).
\end{equation}

To this end, we first apply \cite[Proposition 5.2]{CMN22} to conclude that there exists a constant $C'>0$ such that if $p\geq1$ and $0<\delta<1$ then
\begin{equation}\label{e:uest2}
F_p(\delta):=\sup\Big\{\frac{1}{2p}\,\log\wi P_p(x):\,x\in\wi X,\;
\dist(x,E)\geq\delta\Big\}\leq\frac{C'}{p}\,(1-\log\delta)+\delta+\Omega_{\wi\varphi}(\delta).
\end{equation}
Indeed, as noted in \cite[Remark 5.3]{CMN22}, the upper estimate proved in \cite[Proposition 5.2]{CMN22} holds for the full Bergman kernel $\wi P_p(x)$ of $H^0_{(2)}(\wi X,\wi L^p)$, for points $x$ at distance at least $\delta$ from $E$. This follows from a standard subaverage inequality and the fact that there exists $K>0$ such that 
\[\pi^\star\omega^n(x)\gtrsim\dist(x,E)^K\,\wi\omega^n(x)\,,\,\;
\forall\,x\in\wi X\,.\]

Let $\wi\varphi_p$ be the global Fubini-Study potentials of $\wi\gamma_p$ defined in \eqref{e:wiFSp}. We then proceed exactly as in the proof of \cite[Proposition 5.4]{CMN22} to show that there exists a constant $C''>1$ such that if $p\geq1$ and $0<\delta<1$ then 
\begin{equation}\label{e:uest3}
\wi\varphi_p(x)\leq\wi\varphi(x)+C''\delta+\Omega_{\wi\varphi}(C''\delta)+F_p(\delta/C'')
\end{equation}
holds for all $x\in\wi X$ (see \cite[(5.9]{CMN22}), where $F_p$ is defined in \eqref{e:uest2}. Indeed, \eqref{e:uest3} is proved at points near $E$ by using the fact that $\wi\varphi_p$ is $\wi\alpha$-psh, and by applying the maximum principle on coordinate polydiscs centered at points of $E$ and using the fact that $E$ is a normal crossings divisor.  

We now infer from the definition of $\wi\varphi_\eq$ and \eqref{e:uest3} that
\[\wi\varphi_p\leq\wi\varphi_\eq+C''\delta+\Omega_{\wi\varphi}(C''\delta)+F_p(\delta/C'').\]
Using \eqref{e:uest2} and \eqref{e:tpsih} this implies that 
\[\frac{1}{2p}\,\log\wi P_p\leq\wi\Psi_{\wi h}+
(C''+1)\delta+2\Omega_{\wi\varphi}(C''\delta)+
\frac{C'}{p}\,\Big(1-\log\frac{\delta}{C''}\Big),\]
which yields \eqref{e:uest1}.

\medskip

\textit{Step 2}. We prove here that there exist a constant $C>0$ and $p_0\in\N$ such 
that for all $p\geq p_0$ the following estimate holds on $\wi X\setminus Z$:
\begin{equation}\label{e:lest1}
\frac{1}{2p}\,\log\wi P_p\geq\wi\Psi_{\wi h}+\frac{C}{p}\,\eta.
\end{equation}

Let $H^0_{(2)}(\wi X,\wi L^p,H_p,\wi\omega^n)$ be the Bergman space of 
$L^2$-integrable sections of $\wi L^p$ with respect to the volume form 
$\wi\omega^n$ on $\wi X$ and the metric $H_p$ on $\wi L^p$, where 
\[H_p=\wi h_0^pe^{-2\psi_p},\;\psi_p=(p-p_0)\wi\varphi_\eq+p_0\eta,\]
and $p_0\in\N$ will be specified later. Recall that 
$\wi\varphi_\eq\in\PSH(\wi X,\wi\alpha,\wi\varphi)$ and
$\wi\varphi_\eq\geq\eta-M$ for some constant $M$. Using \eqref{e:eta} 
we see that $\psi_p\leq(p-p_0)\wi\varphi$ if $p\geq p_0$. 
Moreover, 
\[c_1(\wi L^p,H_p)=(p-p_0)(\wi\alpha+dd^c\wi\varphi_\eq)+
p_0(\wi\alpha+dd^c\eta)\geq p_0\varepsilon_0\wi\omega.\]
 
We fix $p_0$ such that $p_0\varepsilon_0\geq1+B+A/2$, 
where $A,B$ are as in Theorem \ref{T:dbar} and let $p\geq p_0$. Applying  
Theorem \ref{T:dbar} for $\wi X,\wi\omega,\wi L^p$, with suitable weights $\chi$ as in the proof of 
\cite[Theorem 5.1]{CM15}, we can show that there exists a constant $C_1>0$ 
such that for all $p\geq p_0$ and $x\in\wi X\setminus Z$ there exists 
$S_x\in H^0_{(2)}(\wi X,\wi L^p,H_p,\wi\omega^n)$ with $S_x(x)\neq0$ and 
\[\|S_x\|^2_{H_p,\wi\omega^n}\leq C_1|S_x(x)|^2_{H_p}\,.\]
Note that $H_p=\wi h^pe^{2F_p}$, where 
$F_p=p\wi\varphi-\psi_p\geq p_0\wi\varphi$. 
Then $F_p\geq ap_0$, where $a:=\min_{\wi X}\wi\varphi$, and by \eqref{e:wiom} we get
\[C_1|S_x(x)|^2_{\wi h^p}e^{2F_p(x)}\geq
\|S_x\|^2_{H_p,\wi\omega^n}=
\int_{\wi X}|S_x|^2_{\wi h^p}e^{2F_p}\,\wi\omega^n\geq 
e^{2ap_0}\|S_x\|^2_{\wi h^p,\pi^\star\omega^n}.\]
Using \eqref{e:Bkvar}, this implies that 
\[\wi P_p(x)\geq C_1^{-1}e^{2ap_0-2F_p(x)}.\]
Note that  
\[F_p=p\wi\varphi-\psi_p=p(\wi\varphi-\wi\varphi_\eq)+
p_0\wi\varphi_\eq-p_0\eta\leq 
-p\wi\Psi_{\wi h}-p_0\eta+p_0\max_{\wi X}\wi\varphi_\eq.\]
Thus 
\[\frac{1}{2p}\,\log\wi P_p\geq\wi\Psi_{\wi h}-\frac{1}{2p}\Big(\log C_1-2ap_0+
2p_0\max_{\wi X}\wi\varphi_\eq\Big)+\frac{p_0}{p}\,\eta\]
holds on $\wi X\setminus Z$ for all $p\geq p_0$. This yields \eqref{e:lest1} 
since $\eta\leq-1$.

\medskip

\textit{Step 3}. We now conclude the proof of the theorem. It follows from 
\eqref{e:Dreg} by using the {\L}ojasiewicz inequality that there exist constants 
$N,M>0$ such that 
\[\eta(x)\geq-N\big|\log\dist\big(x,Z\big)\big|-M\]
holds for all $x\in\wi X$. Hence by \eqref{e:lest1}
\begin{equation}\label{e:lest2}
\frac{1}{2p}\,\log\wi P_p\geq\wi\Psi_{\wi h}-
\frac{C_1}{p}\,\Big(\big|\log\dist\big(x,Z\big)\big|+1\Big)
\end{equation}
holds on $\wi X$ for all $p\geq p_0$ and some constant $C_1>0$.

If $\varphi$ is continuous we infer from \eqref{e:uest1} that given 
$\varepsilon>0$ there exists $p_\varepsilon\in\N$ such that 
\[\frac{1}{2p}\,\log\wi P_p\leq\wi\Psi_{\wi h}+\varepsilon\]
holds on $\wi X$ for $p\geq p_\varepsilon$. Assertion $(i)$ of 
Theorem \ref{T:wiFS} now follows from this and \eqref{e:lest2},
since the function $\log\dist\big(\cdot,Z\big)\in L^1(\wi X,\wi\omega^n)$ 
(see, e.g., \cite[Lemma 5.2]{CMN16}).

If $\varphi$ is H\"older continuous then so is $\wi\varphi$, 
hence $\Omega_{\wi\varphi}(\delta)\leq C_2\delta^\nu$ 
for some constant $C_2>0$, where $\nu$ is the H\"older exponent of 
$\wi\varphi$. Taking $\delta=p^{-1/\nu}$ in \eqref{e:uest1} we see that 
\[\frac{1}{2p}\,\log\wi P_p\leq\wi\Psi_{\wi h}+C\Big(p^{-1/\nu}+
p^{-1}+\frac{\log p}{\nu p}\Big)+2C_2\,\frac{C^\nu}{p}\leq
\wi\Psi_{\wi h}+C_3\,\frac{\log p}{p}\]
holds on $\wi X$ for $p\geq2$, where $C_3>0$ is a constant. Assertion 
$(ii)$ follows immediately from this and \eqref{e:lest2}. This finishes the 
proof of Theorem \ref{T:wiFS}.
\end{proof}

\begin{proof}[Proof of Theorem \ref{T:FS}] 
Since $\pi:\wi X\setminus E\to X\setminus Y$ is a biholomorphism we infer 
using \eqref{e:wiFS}, \eqref{e:tpsih}, \eqref{e:wiom} that 
\begin{equation}\label{e:norm}
\int_X\Big|\frac{1}{2p}\,\log P_p-\Psi_h\Big|\,\omega^n=
\int_{\wi X}\Big|\frac{1}{2p}\,\log\wi P_p-\wi\Psi_{\wi h}\Big|\,\pi^\star\omega^n\leq
\int_{\wi X}\Big|\frac{1}{2p}\,\log\wi P_p-\wi\Psi_{\wi h}\Big|\,\wi\omega^n.
\end{equation}
Theorem \ref{T:wiFS} $(i)$ implies that $\frac{1}{2p}\,\log P_p\to\Psi_h$ in 
$L^1(X,\omega^n)$ as $p\to\infty$. Hence by \eqref{e:FSp1}, \eqref{e:psih}, 
\eqref{e:cu1}, \eqref{e:cu2} we get 
\[\frac{1}{p}\,\gamma_p=c_1(L,h)+\frac{1}{2p}\,dd^c\log P_p\to 
c_1(L,h)+dd^c\Psi_h=c_1(L,h_\eq),\]
weakly on $X$. Finally, assume that $\varphi$ is H\"older continuous. Then 
\eqref{e:sbk} follows from \eqref{e:norm} and Theorem \ref{T:wiFS} $(ii)$, 
since $\log\dist\big(\cdot,Z\big)\in L^1(\wi X,\wi\omega^n)$ 
\cite[Lemma 5.2]{CMN16}.
\end{proof}

\section{Zeros of random sequences of holomorphic sections}\label{S:zeros}

In this section we give the proof of Theorem \ref{T:zeros} and recall from 
\cite{BCM, BCHM} several important classes of probability measures 
which satisfy condition (B). 

\begin{proof}[Proof of Theorem \ref{T:zeros}] 

We use the main ideas from the proof of \cite[Theorem 1.1]{BCM} together with the 
convergence of the Bergman kernels and Fubini-Study currents established in 
Theorem \ref{T:FS}. 

\smallskip

$(i)$ There exists a constant $c>0$ such that, for every smooth real $(n-1,n-1)$ form 
$\phi$ on $X$, the total variation of $dd^c\phi$ satisfies 
$|dd^c\phi|\leq c\|\phi\|_{\cC^2}\,\omega^n$.
If $\phi$ is such a form, using \eqref{e:cbk}, 
we have show that 
\begin{equation}\label{e:expFS}
\frac{1}{p}\,\langle \E[s_p=0]-\gamma_p,\phi\rangle\to0\,,
\,\text{ as $p\to\infty$.}
\end{equation}
For $s_p\in H^0_{(2)}(X,L_p)$ we infer using \eqref{e:LP}
and \eqref{e:FSp1} that
\begin{equation}\label{e:sFS}
\big\langle [s_p=0],\phi\big\rangle=
\big\langle \gamma_p,\phi\big\rangle +
\int_X\log\frac{|s_p|_{h^p}}{\sqrt{P_p}}\,dd^c\phi.
\end{equation}
Let us write 
\begin{equation}\label{e:sp}
s_p=\sum_{j=1}^{d_p}a_jS_j^p.
\end{equation} 
Furthermore, for $x\in X$ we let $e$ 
be a holomorphic frame of $L$ on a neighborhood $U$ of $x$, so
$S_j^p=s_j^pe^{\otimes p}$ with $s_j^p\in\cO_X(U)$. 
Let 
\begin{align}
u^p(x)&=(u_1(x),\dots,u_{d_p}(x)),\; u_j(x)=
\frac{s_j^p(x)}{\sqrt{|s_1^p(x)|^2+\ldots+|s_{d_p}^p(x)|^2}}\,,\label{e:up}\\
\langle a, u^p\rangle&=a_1u_1+\ldots +a_{d_p}u_{d_p}.\label{e:ip}
\end{align}
By H\"older's inequality and assumption (B) we get 
\[\int_{H^0_{(2)}(X,L^p)}\Big|\log\frac{|s_p(x)|_{h^p}}{\sqrt{P_p(x)}}\Big|
\,d\sigma_p(s_p)=\int_{\C^{d_p}}\big|\log|\langle a ,u^p(x)\rangle|\big|
\,d\sigma_p(a)\leq C_p^{1/\nu}.\]
Hence by Tonelli's theorem, 
\[\int_{H^0_{(2)}(X,L^p)}\int_X\Big|\log\frac{|s_p|_{h^p}}{\sqrt{P_p}}\Big|
|dd^c\phi|\,d\sigma_p(s_p)
\leq C_p^{1/\nu}\int_X|dd^c\phi|\leq c\,C_p^{1/\nu}\|\phi\|_{\cC^2}\int_X\omega^n.\]
Using this and \eqref{e:sFS} we infer that $\E[s_p=0]$ is a well-defined positive closed 
current and 
$$ \big|\big\langle \E[s_p=0]-\gamma_p,\phi\big\rangle\big|
\leq c\,C_p^{1/\nu}\|\phi\|_{\cC^2}\int_X\omega^n.$$
Now \eqref{e:expFS} follows since $C_p^{1/\nu}/p\to0$.

\smallskip

$(iii)$ We define 
\[Y_p:\mathcal H\to[0,\infty)\,,\,\;
Y_p(s)=\frac{1}{p}\int_X\Big|\log\frac{|s_p|_{h^p}}{\sqrt{P_p}}\Big|\,\omega^n,\]
where $s=\{s_p\}$. By H\"older's inequality
\[0\leq Y_p(s)^\nu\leq\frac{1}{p^\nu}\left(\int_X\omega^n\right)^{\nu-1}
\int_X\Big|\log\frac{|s_p|_{h^p}}{\sqrt{P_p}}\Big|^\nu\omega^n.\]
Fix $x\in X$ and let $u^p(x)$ be defined in \eqref{e:up}. Using (B) and the 
notation from \eqref{e:sp}, \eqref{e:ip} we obtain
\[\int_{H^0_{(2)}(X,L^p)}\Big|\log\frac{|s_p(x)|_{h^p}}
{\sqrt{P_p(x)}}\Big|^\nu d\sigma_p(s_p)=\int_{\C^{d_p}}
\big|\log|\langle a ,u^p(x)\rangle|\big|^\nu d\sigma_p(a)\leq C_p.\]
Hence by Tonelli's theorem 
\[\int_{\mathcal H} Y_p(s)^\nu d\sigma(s)\leq\frac{1}{p^\nu}
\left(\int_X\omega^n\right)^{\nu-1}\int_X\int_{H^0_{(2)}(X,L^p)}
\Big|\log\frac{|s_p|_{h^p}}{\sqrt{P_p}}\Big|^\nu d\sigma_p(s_p)\;
\omega^n\leq\frac{C_p}{p^\nu}\left(\int_X\omega^n\right)^\nu.\]
It follows that
\[\sum_{p=1}^\infty\int_{\mathcal H} Y_p(s)^\nu d\sigma(s)<\infty,\]
hence 
\[Y_p(s)=\Big\|\frac{1}{p}\,\log|s_p|_{h^p}-\frac{1}{2p}\,\log P_p\Big\|_{L^1(X,\omega^n)}
\to0 \text{ as } p\to\infty,\]
for $\sigma$-a.e.\ $s\in\mathcal{H}$. By \eqref{e:cbk} we conclude that 
$\frac{1}{p}\log|s_p|_{h^p}\to\Psi_h$ in $L^1(X,\omega^n)$ for $\sigma$-a.e.\ 
$s\in\mathcal{H}$.
Using \eqref{e:LP}, this implies that 
\[\frac{1}{p}\,[s_p=0]\to c_1(L,h)+dd^c\Psi_h=c_1(L,h_\eq)\]
weakly on $X$, for $\sigma$-a.e.\ $s\in\mathcal{H}$. 

\smallskip

$(ii)$ There exists an increasing sequence $\{p_j\}$ such that 
$\sum_{j=1}^\infty C_{p_j}p_j^{-\nu}<\infty$. So we can argue as in the 
proof of $(iii)$, working with the sequence $\{p_j\}$.
\end{proof}

\medskip

We review next several classes of probability measures from \cite{BCM,BCHM} 
that verify assumption (B), with estimates for the constants $C_p$. Note that if 
the measures $\sigma_p$ satisfy (B) for some $\nu\geq1$ with 
$C_p=\Gamma_\nu$ independent of $p$, then the hypothesis 
of Theorem \ref{T:zeros} $(i)$ is verified. Moreover, the hypothesis 
of Theorem \ref{T:zeros} $(iii)$ is also verified in this case if $\nu>1$, so 
Theorem \ref{T:zeros} applies to such measures.

As in \cite{BCM,BCHM} we consider measures on $\C^k$, for any  
$k\geq1$. Some of these will be singular measures supported on $\R^k\subset\C^k$ 
or on the unit spheres ${\mathbf S}^{2k-1}\subset\C^k$ and 
${\mathbf S}^{k-1}\subset\R^k\subset{\mathbb C}^k$. We denote in the sequel 
by $\lambda_m$ the Lebesgue measure on $\R^m$. Moreover, we let 
\[\|a\|^2=\sum_{j=1}^m|a_j|^2, \text{ where } a=(a_1,\ldots,a_m)\in\C^m.\]

\medskip

\noindent
\textit{Gaussians}. These are the measures $\sigma_k$ on $\C^k$ and 
$\sigma'_k$ on $\R^k\subset\C^k$ given by the densities 
\begin{equation}\label{e:Gauss}
d\sigma_k(a)=\frac{1}{\pi^k}\,e^{-\|a\|^2}\,d\lambda_{2k}(a),\;a\in\C^k,\;\;
d\sigma'_k(a')=\frac{1}{\pi^{k/2}}\,e^{-\|a'\|^2}\,d\lambda_k(a'),\;a'\in\R^k\subset\C^k.
\end{equation}
Both these measures verify condition (B) for every $\nu\geq1$, with constants 
$C_k=\Gamma_\nu$ independent of $k$ 
(\cite[Lemma 4.8]{BCM}, \cite[Proposition 4.2]{BCHM}).

\medskip

\noindent
\textit{Fubini-Study volumes}. Let $\alpha>0$ and 
$\sigma_{k,\alpha}\,,\sigma'_{k,\alpha}$ be the measures 
given by the radial densities
\begin{equation}\label{e:rad}
\begin{split}
d\sigma_{k,\alpha}(a)&=\frac{\Gamma(k+\alpha)}{\Gamma(\alpha)\pi^k}\,
(1+\|a\|^2)^{-k-\alpha}\,d\lambda_{2k}(a),\;a\in\C^k, \\
d\sigma'_{k,\alpha}(a')&=\frac{\Gamma(\frac{k}{2}+\alpha)}
{\Gamma(\alpha)\pi^{\frac{k}{2}}}\,(1+\|a'\|^2)^{-\frac{k}{2}-\alpha}\,
d\lambda_k(a'),\;a'\in\R^k\subset\C^k,
\end{split}
\end{equation}
where $\Gamma$ is the Gamma function. In particular, the Fubini-Study volume 
on the projective space ${\mathbb P}^k\supset\C^k$ 
is given by the measure $\sigma_{k,1}$ on $\C^k$, see \eqref{e:FSvol}.
Both measures from \eqref{e:rad} verify condition (B) for every $\nu\geq1$, with 
constants $C_k=\Gamma_{\alpha,\nu}$ independent of $k$ 
(\cite[Section 4.2.2]{BCM}, \cite[Proposition 4.3]{BCHM}).

\medskip

\noindent
\textit{Area measure of spheres}. Let ${\mathcal A}_m$ be the surface measure 
on the unit sphere ${\mathbf S}^{m-1}\subset\R^m$, and let
\begin{equation}\label{e:sphere}
\sigma_m=\frac{1}{s_m}\,{\mathcal A}_m\,, \text{ where } 
s_m={\mathcal A}_m({\mathbf S}^{m-1})=
\frac{2\pi^{\frac{m}{2}}}{\Gamma(\frac{m}{2})}\,.
\end{equation} 
The measures $\sigma_{2k}$ on ${\mathbf S}^{2k-1}\subset\C^k$ and 
$\sigma_k$ on ${\mathbf S}^{k-1}\subset\R^k\subset\C^k$ verify condition (B) 
for every $\nu\geq1$, with constants $C_k=M_\nu(\log k)^\nu$ 
(\cite[Lemma 4.11]{BCM}, \cite[Proposition 4.4]{BCHM}).
Recall that by Siegel's lemma we have 
\begin{equation}\label{e:Siegel}
d_p=\dim H^0(X,L^p)\lesssim p^n,
\end{equation}
for any holomorphic line bundle $L$ over a compact, normal, irreducible complex 
space $X$ of dimension $n$. Taking $\nu=2$ it follows that the spherical measures 
$\sigma_{2d_p},\,\sigma_{d_p}$ on $H^0_{(2)}(X,L^p)\cong\C^{d_p}$ satisfy 
assumption (B) with constants $C_p\sim(\log p)^2$.
So the hypothesis 
$\sum_{p=1}^\infty C_pp^{-2}<\infty$ in 
Theorem \ref{T:zeros} $(iii)$ is verified, and 
Theorem \ref{T:zeros} holds for the spherical measures
$\sigma_{2d_p},\,\sigma_{d_p}$.

\medskip

\noindent
\textit{Random holomorphic sections with i.i.d.\ coefficients}.
Let $\phi_1:\C\to[0,M]$, $\phi_2:\R\to[0,M]$, where $M>0$, be measurable 
functions such that $\int_\C\phi_1\,d\lambda_2=1$, 
$\int_\R\phi_2\,d\lambda_1=1$. Moreover, assume that there exist 
constants $c>0$, $\rho>1$ such that for all $R>0$, 
\begin{equation}\label{e:tail} 
\int_{\big\{\zeta\in\C:\,|\zeta|>e^R\big\}}\phi_1(\zeta)\,d\lambda_2(\zeta)\leq cR^{-\rho}
\,,\,\;\int_{\big\{x\in\R:\,|x|>e^R\big\}}\phi_2(x)\,d\lambda_1(x)\leq cR^{-\rho}.
\end{equation}
Let $\sigma_k\,,\sigma'_k$ be the measures with densities
\begin{equation}\label{e:iid}
\begin{split}
d\sigma_k(a)&=\phi_1(a_1)\ldots\phi_1(a_k)\,d\lambda_{2k}(a),\;a=(a_1,\ldots,a_k)\in\C^k,\\
d\sigma_k(a')&=\phi_2(a'_1)\ldots\phi_2(a'_k)\,d\lambda_k(a'),\;a'=(a'_1,\ldots,a'_k)\in\R^k\subset\C^k.
\end{split}
\end{equation}
These measures verify condition (B) for every $1\leq\nu<\rho$ with constants 
$C_k=\Gamma k^{\nu/\rho}$, where $\Gamma=\Gamma(M,c,\rho,\nu)>0$ 
(\cite[Lemma 4.15]{BCM}, \cite[Proposition 4.6]{BCHM}).

Endowing the Bergman spaces $H^0_{(2)}(X,L^p)$ with the measures in \eqref{e:iid} 
means the following: if $\{S_j^p\}_{j=1}^{d_p}$ is a fixed orthonormal basis of 
$H^0_{(2)}(X,L^p)$, then random holomorphic sections are of the form 
\[s_p=\sum_{j=1}^{d_p}a_j^pS_j^p\,,\]
where the coefficients $\{a_j^p\}_{j=1}^{d_p}$ are i.i.d.\ complex, resp.\ real, 
random variables with density function $\phi_1$, resp.\ $\phi_2$.
Using \eqref{e:Siegel}, we infer that the measures $\sigma_{d_p},\,\sigma'_{d_p}$, 
defined by \eqref{e:iid} on $H^0_{(2)}(X,L^p)\cong\C^{d_p}$, satisfy 
assumption (B) for each $1\leq\nu<\rho$, with constants $C_p\sim p^{n\nu/\rho}$. 
Assume that $\rho>n+1$ and fix $\nu$ such that 
\[\frac{\rho}{\rho-n}<\nu<\rho,\; \text{ so } C_pp^{-\nu}\sim p^{-(\rho-n)\nu/\rho}.\]
It follows that the hypothesis 
$\sum_{p=1}^\infty C_pp^{-\nu}<\infty$ in Theorem \ref{T:zeros} $(iii)$ is verified. 
Hence Theorem \ref{T:zeros} applies to the measures $\sigma_{d_p},\,\sigma'_{d_p}$ 
from \eqref{e:iid}, provided that the tail decay condition \eqref{e:tail} holds with 
exponent $\rho>n+1$.

\medskip

We conclude the paper by specializing our results to the case of polynomials 
on $\C^n$.

\begin{Example}\label{E:poly} 
Let $X={\mathbb P}^n$, $\omega_{\FS}$ be the Fubini-Study 
form on $\mathbb{P}^n$, $L=\mO(1)$ be the hyperplane line bundle, and 
$h_{\FS}$ be the Fubini-Study metric on $\mO(1)$. So 
$c_1(\mO(1),h_{\FS})=\omega_{\FS}$. We denote by 
$[z_0:\ldots:z_n]$ the homogeneous coordinates on $\P^n$. 

The global holomorphic sections of $L^p=\mO(p)$ are given by homogeneous 
polynomials of degree $p$ in $\C[z_0,\ldots,z_n]$. Let 
$U_0=\{[1:\zeta]\in\mathbb{P}^n:\zeta\in\C^n\}\cong\C^n$. 
Using the holomorphic frame $e_0$ of $\mO(1)$ on $U_0$ that corresponds to $z_0$, 
we can identify $H^0(\mathbb{P}^n,\mO(p))$ to the space of polynomials on 
$\C^n$ of degree at most $p$,
\begin{equation}\label{e:hip}
H^0(\mathbb{P}^n,\mO(p))\to\C_p[\zeta]=
\left\{f\in\C[\zeta_1,\ldots,\zeta_n]:\deg f\leq p\right\}\,,\:\:s\mapsto s/z_0^p.
\end{equation}
Writing $\|\zeta\|^2=\sum_{j=1}^n|\zeta_j|^2$ for 
$\zeta=(\zeta_1,\ldots,\zeta_n)\in\C^n$, we have 
\begin{align}
|e_0|_{h_\FS}&=e^{-\rho}, \text{ where } \rho(\zeta)=
\frac{1}{2}\,\log\Big(1+\|\zeta\|^2\Big), \label{e:rho} \\
\omega_\FS^n|_{U_0}&=\frac{n!}{\pi^n}\,\frac{1}{(1+\|\zeta\|^2)^{n+1}}\,
d\lambda_{2n}(\zeta).\label{e:FSvol}
\end{align}

It is well known that the class $\PSH(\P^n,\omega_\FS)$ is in one-to-one 
correspondence to the Lelong class $\mathcal{L}({\mathbb C}^n)$ of 
entire psh functions with logarithmic growth (see, e.g., 
\cite[Section 2]{GZ05}, \cite[Example 4.4]{BCM}), 
\[\mathcal{L}({\C}^n)=\left\{v\in\PSH(\C^n):\,\exists\, C_v\in\R \text{ such that
$v(z)\leq \log^+\|z\|+C_v$ on $\C^n$}\right\}.\]

Let $h=h_\FS e^{-2\varphi}$ be a continuous Hermitian metric on $\mO(1)$, 
where $\varphi$ is a continuous function on $\P^n$.
Let $h_\eq=h_\FS e^{-2\varphi_\eq}$ be the equilibrium metric
associated to $h$, as defined in \eqref{e:eqm}, \eqref{e:eqw}.
Then the local weights $\psi$, resp.\ $\psi_\eq$, of $h$ and
$h_\eq$ on $U_0\cong\C^n$ with respect to the frame $e_0$
are given by 
\[\psi=\rho+\varphi,\;
\psi_\eq=\rho+\varphi_\eq=\sup\{v\in\mathcal{L}({\C}^n):
\,v\leq\psi \text{ on } \C^n\}.\]
Thus $\psi_\eq$ is the largest psh function with logarithmic growth
dominated by $\psi$ on $\C^n$. 
If $s_p\in H^0(\mathbb{P}^n,\mO(p))$ corresponds to
$f_p\in\C_p[\zeta]$ then $|s_p|_{h^p}=|f_p|e^{-p\psi}$ on $U_0$.
Moreover, $\Psi_h=\psi_\eq-\psi$ on $U_0$,
where $\Psi_h$ is defined in \eqref{e:psih}.

The map \eqref{e:hip} induces an isometry
between $H^0_{(2)}(\mathbb{P}^n,\mO(p))$ and the space 
\[\C_{p,(2)}[\zeta]=\left\{f\in\C_p[\zeta]:\,
\int_{\C^n}|f|^2e^{-2p\psi}\,\frac{\omega^n_{\FS}}{n!}<\infty\right\}.\]
Let $\sigma_p$ be probability measures on $\C_{p,(2)}[\zeta]$ and set 
$(\mathcal{H},\sigma)=
\left(\prod_{p=1}^\infty \C_{p,(2)}[\zeta],\prod_{p=1}^\infty\sigma_p\right)$.
Theorem \ref{T:zeros} has the following immediate consequence.

\begin{Corollary}\label{C:poly}
Let $\psi:\C^n\to\R$ be a continuous function such that $\psi-\rho$ extends 
continuously to $\P^n\supset\C^n$, where $\rho$ is as in \eqref{e:rho}, and set 
\[\psi_\eq(\zeta)=\sup\{v(\zeta):\,v\in\mathcal{L}({\C}^n),\,
v\leq\psi \text{ on } \C^n\},\;\zeta\in\C^n.\]
Assume that $\sigma_p$ are probability measures on $\C_{p,(2)}[\zeta]$ that 
satisfy condition (B),  and that $\sum_{p=1}^{\infty}C_p\,p^{-\nu}<\infty$. Then 
for $\sigma$-a.e.\ sequence $\{f_p\}\in\mathcal{H}$ we have as $p\to\infty$,
\[\frac{1}{p}\,\log|f_p|\to\psi_\eq\text{ in } L^1(\C^n,\omega_{\FS}^n),\;
\frac{1}{p}\,[f_p=0]\to dd^c\psi_\eq\, \text{ weakly on } \C^n.\]
\end{Corollary}
\end{Example}

We refer to \cite{Bl05,BL15,Ba1} for related results when the inner products on $H^0(X,L^p)$ 
are defined using certain classes of Bernstein-Markov measures in place of the 
volume form $\omega^n$ (see also the survey \cite{BCHM}).

\end{document}